\documentclass[11pt,fleqn]{article}
\usepackage{amsmath,amsfonts,amsthm,amssymb}
\newcommand{\R}{\mathbb R}
\newcommand{\C}{\mathbb C}

\renewcommand{\Re}{\mathop {\rm Re}\nolimits}

  \newtheorem{theorem}{Theorem}[section]
  \newtheorem{corollary}[theorem]{Corollary}
  \newtheorem{lemma}[theorem]{Lemma}

\begin{document}

\title{Numerical radius and distance from unitary operators}

\author{Catalin {\sc Badea} and Michel {\sc Crouzeix} }

\date{}

\maketitle
\begin{abstract}
Denote by $w(A)$ the numerical radius of a bounded linear operator $A$ acting on Hilbert space. 
Suppose that $A$ is invertible and that
$w(A)\leq 1{+}\varepsilon$ and  $w(A^{-1})\leq 1{+}\varepsilon$ for some $\varepsilon\geq0$. 
It is shown that
$\inf\{\|A{-}U\|\,: U$ unitary$\}\leq c\varepsilon^{1/4}$ for some constant $c>0$. This generalizes a 
result due to J.G.~Stampfli, which is obtained for $\varepsilon = 0$. 
An example is given showing that the exponent $1/4$ is optimal. The more general case of the operator $\rho$-radius 
$w_{\rho}(\cdot)$ is discussed for $1\le \rho \le 2$.
\end{abstract}

\section{Introduction and statement of the results}
Let $H$ be a complex Hilbert space endowed with the inner product $\langle\cdot,\cdot\rangle$
and the associated norm $\|\cdot\|$. We denote by ${\cal B}(H)$ the C$^\ast$-algebra of all
bounded linear operators on $H$ equipped with the operator norm
\begin{equation*}
\|A\|=\sup\{\|Ah\|\,: h\in H, \|h\|=1\}.
\end{equation*}
It is easy to see that unitary operators can be characterized as 
invertible contractions with contractive inverses, i.e. as operators $A$ with $\|A\| \le 1$ and  
$\|A^{-1}\| \le 1$. More generally, if $A\in{\cal B}(H)$ is invertible then 
\begin{equation*} \inf\left\{\|A{-}U\|\,: U \textrm{ unitary }\right\} 
= \max\left(\|A\| - 1, 1 - \frac{1}{\|A^{-1}\|}\right) .
\end{equation*}
We refer to \cite[Theorem 1.3]{rogers} and \cite[Theorem 1]{terElst} for a proof of this equality 
using the polar decomposition of 
bounded operators. It also follows from this proof that if $A\in{\cal B}(H)$ is an invertible operator 
satisfying $\|A\|\leq r$ and $\|A^{-1}\|\leq r$ for some $r \ge 1$, 
then there exists a unitary operator $U\in{\cal B}(H)$ such that $\|A{-}U\|\leq r{-}1$.\medskip

The numerical radius of the operator $A$ is defined by
\begin{equation*}
w(A)=\sup\{|\langle Ah,h\rangle|\,: h\in H, \|h\|=1\}.
\end{equation*}
Stampfli has proved in \cite{stampfli} that numerical radius contractivity of $A$ 
and of its inverse $A^{-1}$, that is $w(A)\leq 1$ and $w(A^{-1}) \leq 1$, imply that
$A$ is unitary. 
We define a function $\psi(r)$ for $r\geq1$ by
\begin{equation*}
\psi(r)=\sup\{\|A\|\,: A\in{\cal B}(H), w(A)\leq r, \ w(A^{-1})\leq r\},
\end{equation*}
the supremum being also considered over all Hilbert spaces $H$.
Then the conditions $w(A)\leq r$ and  $w(A^{-1})\leq r$ imply 
$\max(\|A\|{-}1, 1{-}\|A^{-1}\|^{-1})\leq\max(\|A\|{-}1,\|A^{-1}\|{-}1)\leq\psi(r){-}1$, hence the existence of 
a unitary operator $U$ such that $\|A{-}U\|\leq \psi(r){-}1$. 
We have the two-sided estimate 
\begin{equation*}
r{+}\sqrt{r^2-1}\leq\psi(r)\leq2r.
\end{equation*}
The upper bound follows from 
the well-known inequalities $w(A)\leq\|A\|\leq 2w(A)$, while the lower bound is obtained by choosing $H=\C^2$ and 
\begin{equation*}
A= \begin{pmatrix}1&2y\\0&-1\end{pmatrix}\quad\text{with }\ y=\sqrt{r^2\!-\!1}, 
\end{equation*}
in the definition of $\psi$. Indeed, 
we have $A=A^{-1}$, $w(A)=\sqrt{1+y^2}=r$, and $\|A\|=y+\sqrt{1\!+\!y^2}=r+\sqrt{r^2\!-\!1}$.\medskip

Our first aim is to improve the upper estimate.
\begin{theorem}\label{thm:1}
Let $r \ge 1$. Then
\begin{equation}\label{psi}
\psi(r)\leq X(r)+\sqrt{X(r)^2{-}1},\quad\hbox{with}\quad
X(r)=r+\sqrt{r^2-1}.
\end{equation}
\end{theorem}
The estimate given in Theorem \ref{thm:1} is more accurate than $\psi(r)\leq2r$ for 
$r$ close to $1$, more precisely for 
$1\leq r\leq 1.0290855 \dots$. It also gives $\psi(1) = 1$ (leading to Stampfli's result)
and the following asymptotic estimate.

\begin{corollary}\label{cor:1}
 We have 
$$\psi(1{+}\varepsilon)\leq 1+\sqrt[4]{8\varepsilon}
+O(\varepsilon^{1/2}), \quad \varepsilon \to 0 .$$
\end{corollary}

Our second aim is to prove that the exponent $1/4$ in Corollary \ref{cor:1} is optimal. 
This is a consequence of the following result.

\begin{theorem}\label{thm:2}
 Let $n$ be a positive integer of the form $n = 8k+4$. There exists a $n\times n$ 
invertible matrix $A_n$ with complex 
entries such that 
$$ w(A_n) \le \frac{1}{\cos\frac{\pi}{n}}, \quad w(A_n^{-1}) \le\frac{1}{\cos\frac{\pi}{n}}, \quad \|A_n\| = 1+\frac{1}{8\sqrt{n}} .$$
\end{theorem}
Indeed, Theorem \ref{thm:2} implies that 
$$\psi\big(\frac{1}{\cos\frac{\pi}{n}}\big)\geq \|A_n\|=1\!+\!\frac{1}{8\sqrt{n}} .$$
Taking $1\!+\!\varepsilon=1/\cos\frac\pi n=1+ \frac{\pi^2}{2n^2}+O(\frac{1}{n^4})$, we 
see that the exponent $\frac14$ cannot be improved.

\medskip

More generally, we can consider for $\rho\geq1$ the $\rho$-radius $w_\rho(A)$ introduced by
Sz.-Nagy  and Foia\c{s} (see \cite[Chapter 1]{nafo} and the references therein). 
Consider the class ${\mathcal C}_\rho$ of operators 
$T\in {\mathcal B}(H)$
which admit unitary $\rho$-dilations, i.e. there exist a super-space ${\mathcal H}\supset H$ and 
a unitary operator $U\in  {\mathcal B}({\mathcal H})$ such that 
\begin{equation*}
T^n=\rho P U^nP^*,\qquad \hbox{for } n=1,2,\dots .
\end{equation*}
Here $P$ denotes the orthogonal projection from ${\mathcal H}$ onto $H$. 
Then the operator $\rho$-radius is defined by
\begin{equation*}
w_\rho(A)=\inf\{\lambda>0\,; \lambda^{-1}A\in {\mathcal C}_\rho\}. 
\end{equation*}
From this definition it is easily seen that $r(A)\leq w_\rho(A)\leq \rho \|A\|$, 
where $r(A)$ denotes the spectral radius of $A$. Also, $w_\rho(A)$ 
is a non-increasing function of $\rho$. 
Another equivalent definition follows from \cite[Theorem 11.1]{nafo}:
\begin{align*}
w_\rho(A)&=\sup_{h\in \mathcal E_\rho}\Big\{(1\!-\!\tfrac1\rho)\,|\langle Ah,h\rangle|+\sqrt
{(1\!-\!\tfrac1\rho)^2|\langle Ah,h\rangle|^2+(\tfrac2\rho{-}1)\,\| Ah\|^2}\Big\},\quad\hbox{with}\\
 \mathcal E_\rho&=\{h\in H\,;\|h\|=1\ {\rm and} 
 (1\!-\!\tfrac1\rho)^2|\langle Ah,h\rangle|^2
 -(1\!-\!\tfrac2\rho)\,\| Ah\|^2\geq0\}.
\end{align*}
Notice that $ \mathcal E_\rho=\{h\in H\,; \|h\|=1\}$ whenever $1\leq \rho\leq 2$. 
This shows that $w_1(A)=\|A\|$,
$w_2(A)=w(A)$ and $w_\rho(A)$ is a convex function of $A$ if $1\leq \rho\leq 2$. 

We now define a function $\psi_\rho(r)$ for $r\geq1$ by
\begin{equation*}
\psi_\rho(r)=\sup\{\|A\|\,; A\in{\cal B}(H), w_\rho(A)\leq r, \ w_\rho(A^{-1})\leq r\}.
\end{equation*}
As before, the conditions $w_\rho(A)\leq r$ and  $w_\rho(A^{-1})\leq r$ imply the existence 
of a unitary operator $U$ such that $\|A{-}U\|\leq \psi_\rho(r){-}1$, 
and we have $\psi_\rho(r)\leq \rho r$. 
We  will generalize the estimate \eqref{psi} from Theorem \ref{thm:1} 
by proving, for $1\leq\rho\leq2$, the following result.
\begin{theorem}\label{thm:3}
For $1\leq\rho\leq2$ we have
\begin{align}\label{psirho}
\psi_\rho(r)&\leq  X_\rho(r)+\sqrt{X_\rho(r)^2-1},\\
\hbox{with}\quad X_\rho(r)&=\frac{2+\rho r^2-\rho+
\sqrt{(2+\rho r^2-\rho)^2-4r^2}}{2r}.\nonumber
\end{align}
\end{theorem}
\begin{corollary}\label{cor:2}
For $1\leq\rho\leq2$ we have 
$$\psi_\rho(1{+}\varepsilon)\leq 1+\sqrt[4]{8(\rho-1)\varepsilon}+O(\varepsilon^{1/2}), 
\quad \varepsilon \to 0.$$
\end{corollary}
We recover in this way for $1\leq\rho\leq2$ the recent result of Ando and
Li \cite[Theorem 2.3]{anli}, namely that $w_\rho(A)\leq1$ and $w_\rho(A^{-1}) \leq 1$ imply
$A$ is unitary. The range $1\leq\rho\leq2$ coincides with the range of 
those $\rho \ge 1$ for which $w_\rho(\cdot)$ is a norm. Contrarily to \cite{anli}, we have
not been able to treat the case $\rho>2$.

\medskip

The organization of the paper is as follows. In Section 2 we prove Theorem \ref{thm:3}, which reduces to
Theorem \ref{thm:1} in the case $\rho=2$. The proof of Theorem \ref{thm:2}
 which shows the optimality of the exponent $1/4$ in Corollary \ref{cor:1} is given in Section 3. 

As a concluding remark, we would like to mention that 
the present developments have been influenced by the 
recent work of Sano/Uchiyama \cite{saus} and Ando/Li \cite{anli}.
In \cite{crou}, 
inspired by the paper of Stampfli \cite{stampfli}, we have developed another 
(more complicated) approach in the case $\rho=2$.

\section{Proof of Theorem \ref{thm:3} about $\psi_\rho$} 
Let us consider $M=\frac{1}{2}(A+(A^*)^{-1})$; then
\begin{equation*}
M^*M-1=\tfrac{1}{4}(A^*A+(A^*A)^{-1}-2)\geq 0 .
\end{equation*}
This implies $\|M^{-1}\|\leq 1$. In what follows $C^{1/2}$ will denote the positive 
square root of the self-adjoint positive operator $C$. 
The relation $(A^*A-2M^*M+1)^2=4M^*M(M^*M-1)$ yields 
\begin{align*}
A^*A-2M^*M+1&\leq2(M^*M)^{1/2}(M^*M-1)^{1/2},\\\hbox{whence}\quad
A^*A&\leq ((M^*M)^{1/2}+(M^*M-1)^{1/2})^2 .
\end{align*} 
Therefore $\|A\|\leq \|M\|+\sqrt{\|M\|^2-1}$.\medskip

We now assume $1\leq\rho\leq2$. Then $w_\rho(.)$ is a norm and the two 
conditions $w_\rho(A)\leq r $ and $w_\rho(A^{-1})\leq r$
imply $w_\rho(M)\leq r$. The desired estimate of $\psi_\rho(r)$ will follow from the following 
auxiliary result.

\begin{lemma}
Assume $\rho\geq1$. Then the assumptions $w_\rho(M)\leq r$ and $\|M^{-1}\|\leq 1$ imply
$\|M\|\leq X_\rho(r)$.
\end{lemma}

\begin{proof} The contractivity of $M^{-1}$ implies
\begin{equation}\label{a}
\|u\| \leq \|Mu\|, \quad ( \forall u\in H). 
\end{equation}
 As $w_\rho(M)\leq r$, it follows from a generalization by Durszt \cite{durszt} of a decomposition due to Ando \cite{ando}, that 
the operator $M$ can be decomposed as
 \begin{equation*}
 M=\rho r \,B^{1/2}UC^{1/2},
 \end{equation*}
with $U$ unitary, $C$ selfadjoint satisfying $0<C<1$, and $B=f(C)$
with $f(x)=(1\!-\!x)/(1\!-\!\rho(2\!-\!\rho)x)^{-1}$. 
Notice that $f$ is a decreasing function on the segment
$[0,1]$ and an involution\,: $f(f(x))=x$.
Let $[\alpha,\beta]$ be the smallest segment containing the spectrum of $C$. Then
 $[\sqrt \alpha,\sqrt \beta]$ is the smallest segment containing the spectrum of $C^{1/2}$ and  
 $[\sqrt{f(\beta)},\sqrt{f(\alpha)}]$ is the smallest segment containing the spectrum of $B^{1/2}$.
We have
\begin{equation*}
\|u\|\leq \|Mu\|\leq \rho r\,\sqrt{f(\alpha)}\|C^{1/2}u\|,\quad ( \forall u\in H). 
\end{equation*}
Choosing a sequence $u_n$ of norm-one vectors ($\|u_n\|=1$) such that $\|C^{1/2}u_n\|$ tends to
$\sqrt{\alpha}$, we first get\\
$1\leq \rho r\sqrt{\alpha f(\alpha)}$, i.e. $1-(2{+}\rho r^2 {-}\rho)\rho\alpha+\rho^2r^2\alpha^2\leq0$.
Consequently we have
\begin{equation*}
\frac{2{+}\rho r^2{-}\rho-\sqrt{(2{+}\rho r^2{-}\rho)^2{-}4r^2}}{2\,\rho\, r^2}\leq \alpha\leq 
\frac{2{+}\rho r^2{-}\rho+\sqrt{(2{+}\rho r^2{-}\rho)^2{-}4r^2}}{2\,\rho\, r^2},
\end{equation*}
and by $\alpha=f(f(\alpha))$
\begin{equation*}
\frac{2{+}\rho r^2{-}\rho-\sqrt{(2{+}\rho r^2{-}\rho)^2{-}4r^2}}{2\,\rho\, r^2}\leq f(\alpha)\leq 
\frac{2{+}\rho r^2{-}\rho+\sqrt{(2{+}\rho r^2{-}\rho)^2{-}4r^2}}{2\,\rho\, r^2}.
\end{equation*}
Similarly, noticing that $\|(M^*)^{-1}\|\leq 1$, $M^*=\rho r \,C^{1/2}U^*B^{1/2}$ and $C=f(B)$,
we obtain
\begin{equation*}
\frac{2{+}\rho r^2{-}\rho-\sqrt{(2{+}\rho r^2{-}\rho)^2{-}4r^2}}{2\,\rho\, r^2}\leq \beta\leq 
\frac{2{+}\rho r^2{-}\rho+\sqrt{(2{+}\rho r^2{-}\rho)^2{-}4r^2}}{2\,\rho\, r^2}.
\end{equation*}
Therefore
\begin{equation*}
\|M\|\leq \rho r\,\|B^{1/2}\|\,\|C^{1/2}\|
=\rho r\,\sqrt{ f(\alpha)\beta}\leq
\frac{2{+}\rho r^2{-}\rho+\sqrt{(2{+}\rho r^2{-}\rho)^2{-}4r^2}}{2\, r}.
\end{equation*}
This shows that $\|M\|\leq X_\rho(r)$.
\end{proof}

\section{ The exponent $1/4$ is optimal (Proof of Theorem \ref{thm:2})}

Consider the family of $n\times n $ matrices
 $A=DBD$, defined for $n=8k+4$, by
\begin{align*}
D&={\rm diag}(e^{i\pi/2n},\dots,e^{(2\ell-1) i\pi/2n},\dots,e^{(2n-1)i\pi/2n}),\\
B&=I+\tfrac{1}{2\,n^{3/2}}E,\quad\hbox{where $E$ is a matrix whose entries are defined as  }\\
&\quad e_{ij}=1\ \ \hbox{if } 3k+2\leq|i-j|\leq 5k+2,\qquad e_{ij}=0\ \ \hbox{otherwise. }
\end{align*}
We first remark that $\|A\|=\|B\|=1\!+\!\frac{1}{8\sqrt{n}}$. 
Indeed, $B$ is a symmetric matrix with non negative entries,  
$Be=(1\!+\!\frac{1}{8\sqrt{n}})e$ with $e^T=(1,1,1\dots,1)$. 
Thus $\|B\|=r(B)=1\!+\!\frac{1}{8\sqrt{n}}$ by the Perron-Frobenius theorem.\medskip

Consider now the permutation matrix $P$ defined by $p_{ij}=1$ if $i=j\!+\!1$ modulo $n$ and
$p_{ij}=0$ otherwise and the diagonal matrix $\Delta=$diag$(1,\dots,1,-1)$. Then $P^{-1}DP= e^{i\pi/n}\Delta D$ and 
$P^{-1}EP=E$, whence  $(P\Delta)^{-1}AP\Delta=e^{2i\pi/n}A$. Since $P\Delta$ is a unitary matrix, 
the numerical range 
$W(A)=\{\langle Au,u\rangle,; \|u\|=1\}$ of $A$ satisfies $W(A)=W((P\Delta)^{-1}AP\Delta)=e^{2i\pi/n}W(A)$.
This shows that the numerical range of $A$ 
is invariant by the rotation of angle $2\pi/n$ centered in 0, 
and the same property also holds for the numerical range of $A^{-1}$.\medskip

We postpone the proof of the estimates 
$\big\|\frac12(A\!+\!A^*)\big\|\leq 1$ and $\big\|\frac12(A^{-1}\!+\!(A^{-1})^{*})\big\|\leq 1$ 
to later sections. Using these estimates, we obtain that the numerical range $W(A)$ 
is contained in the half-plane $\{z\,;\Re z\leq1\}$,
whence in the regular $n$-sided polygon  given by the intersection of the 
half-planes $\{z\,;\Re (e^{2i\pi k/n}z)\leq1\}$,
$k=1,\dots,n$. Consequently $w(A)\leq 1/\cos(\pi/n)$. The proof 
of $w(A^{-1})\leq 1/\cos(\pi/n)$ is similar. 

\subsection{Proof of
$\big\|\frac12(A\!+\!A^*)\big\|\leq 1$.}
Since the $(\ell,j)$-entry of $A$ is $e^{(\ell+j-1)i\frac\pi n}\Big(\delta_{\ell,j}+\frac{e_{\ell,j}}{2n^{3/2}}\Big)$, the matrix $\frac12(A+A^*)$ is a real symmetric matrix whose $(i,j)-$entry is $\cos\big((i{+}j{-}1)\frac{\pi} {n}\big)\Big(\delta_{i,j}+\frac{e_{i,j}}{2n^{3/2}}\Big)$.
It suffices to show that, for every $u =(u_1,\cdots,u_n)^{\scriptscriptstyle T}\in \R^n$, 
we have $\|u\|^2-\Re \langle A u,u\rangle\geq0$.
Let ${\mathcal E}=\{(i,j)\,;1\leq i,j\leq n, 3k+2\leq|i-j|\leq5k+2\}$. 
The inequality which has to be proved is equivalent to
\begin{align*}
\sum_{i=1}^n 2\sin^2((i\!-\!\tfrac12)\tfrac\pi n)\,u_i^2-\tfrac{1}{2\,n^{3/2}}\sum_{i,j\in{\mathcal E}}\cos((i\!+\!j\!-\!1)\tfrac\pi n)u_i\,u_j\geq0.
\end{align*}
Setting $v_j=u_j\sin((j\!-\!\tfrac12)\tfrac\pi n)$, this may be also written as follows
\begin{align}\label{yy}
2\|v\|^2-\langle Mv,v\rangle+\tfrac{1}{2n^{3/2}}\,\langle Ev,v\rangle\geq0,\qquad (v\in\R^n).
\end{align}
Here $M$ is the matrix whose entries are defined by 
\begin{align*}
m_{ij}=\tfrac{1}{2n^{3/2}}\,\cot((i\!-\!\tfrac12)\tfrac\pi n)\,\cot((j\!-\!\tfrac12)\tfrac\pi n),\quad
\hbox{ if }(i,j)\in \mathcal E,\quad m_{ij}=0\quad\hbox{otherwise.}
\end{align*}
We will see that the Frobenius (or Hilbert-Schmidt) norm of $M$ 
satisfies $\|M\|_{_{F}}\leq \sqrt{9/32}<3/4$. A fortiori, the operator norm of $M$ 
satisfies $\|M\|\leq \tfrac{3}{4}$. Together with $\|E\|= n/4$, this shows that
$\|M\|+\tfrac{1}{2n^{3/2}}\|E\| \leq\tfrac{7}{8}$. Property \eqref{yy} is now verified.
\medskip

It remains to show that $\|M\|_{_{F}}^2\leq \tfrac{9}{32}$. First we notice 
that $m_{ij}=m_{ji}=m_{n+1-i,n+1-j}$, and $m_{ii}=0$. Hence, with $\mathcal E'=\{(i,j)\in \mathcal E\,; i<j\ {\rm and }\ i+j\leq n+1\}$,
\begin{align*}
\|M\|_{_{F}}^2=2\sum_{i<j}|m_{ij}|^2\leq 4\sum_{(i,j)\in\mathcal E'}|m_{ij}|^2.
\end{align*}
We have, for $(i,j)\in \mathcal{E}'$, 
\begin{align*}
2j&\leq i+j+5k+2\leq n+5k+3=13k+7,\quad\hbox{thus}\quad 3k+3\leq  j\leq\tfrac{13k+7}{2},\\
2i&\leq i+j-3k-2\leq n-3k-1=5k+3,\quad\hbox{thus}\quad 1\leq i\leq\tfrac{5k+3}{2}.
\end{align*}
This shows that
\begin{align*}
\tfrac{3\pi}{16}\leq\tfrac{3k+2}{16k+8}\pi\leq (j\!-\!\tfrac12)\tfrac{\pi}{n}\leq\tfrac{13k+6}{16k+8}\pi\leq\pi-\tfrac{3\pi}{16},\quad\hbox{hence}\quad |\cot((j\!-\!\tfrac12)\tfrac{\pi}{n})|\leq\cot\tfrac{3\pi}{16}\leq \tfrac32.
\end{align*}
We also use the estimate $\cot((i\!-\!\tfrac12)\tfrac\pi n)\leq n/(\pi(i\!-\!\tfrac12))$ 
and the classical relation $\sum_{i\geq1}(i-1/2)^{-2}=\pi^2/2$ to obtain
\begin{align*}
\|M\|_{_{F}}^2\leq4\sum_{(i,j)\in\mathcal E'}|m_{ij}|^2\leq \frac{4}{4n^3}\ \frac{n^2}{\pi^2}
\sum_{i\geq1} \frac{1}{(i-1/2)^2}\ (2k\!+\!1)\ \frac{9}{4}=\frac{9}{32}.
\end{align*}

\subsection{Proof of $\big\|\frac12(A^{-1}\!+\!(A^{-1})^{*})\big\|\leq 1$.}

We start from
\begin{align*}
(A^{-1})^{*}
&=D(1+\tfrac{1}{2n^{3/2}}E)^{-1}D\\&=D^2-\tfrac{1}{2n^{3/2}}DED+
\tfrac{1}{4n^3}DE^2(1+\tfrac{1}{2n^{3/2}}E)^{-1}D,
\end{align*} 
and we want to show that $\|u\|^2-\Re \langle A^{-1} u,u\rangle\geq0$. As previously, we set
$v_j=u_j\sin((j\!-\!\tfrac12)\tfrac\pi n)$. 
The inequality $\big\|\frac12(A^{-1}\!+\!(A^{-1})^{*})\big\|\leq 1$ is equivalent to
\begin{align*}
2\|v\|^2-\langle (M_1+M_2+M_3+M_4)v,v\rangle \geq0,\qquad (v\in\R^n).
\end{align*}
Here the entries of the matrices $M_p$, $1\le p \le 4$, are given by
\begin{align*}
(m_1)_{ij}&= -\tfrac{1}{2n^{3/2}}\,\cot((i\!-\!\tfrac12)\tfrac\pi n)\,\cot((j\!-\!\tfrac12)\tfrac\pi n)e_{ij},\\
(m_2)_{ij}&= \tfrac{1}{2n^{3/2}}\,e_{ij},\\
(m_3)_{ij}&=\tfrac{1}{4n^3}\,\cot((i\!-\!\tfrac12)\tfrac\pi n)\,\cot((j\!-\!\tfrac12)\tfrac\pi n) f_{ij},\\
(m_4)_{ij}&= -\tfrac{1}{4n^3}\,f_{ij},
\end{align*}
$e_{ij}$ and $f_{ij}$ respectively denoting the entries of the matrices
$E$ and $F=E^2(1+\tfrac{1}{2n^{3/2}}E)^{-1}$.
Noticing that $M_1=-M$, we have $\|M_1\|\leq \tfrac{3}{4}$, $\|M_2\|= \tfrac{1}{8\sqrt n}$,
 $\|F\|\leq\tfrac{n^2/16}{1-1/(8\sqrt n)}\leq\tfrac{n^2}{14}$ and $\|M_4\|= \tfrac{1}{4n^3}\,\|F\|$. 
Now we use
\begin{align*}
\|M_3\|^2\leq \|M_3\|_{_{F}}^2\leq \tfrac{1}{16n^6}\max_{ij}|f_{ij}|^2\sum_{i,j}|\cot((i\!-\!\tfrac12)\tfrac\pi n)|^2\,|\cot((j\!-\!\tfrac12)\tfrac\pi n)|^2,
\end{align*}
together with
\begin{align*}
\sum_{i,j}|\cot((i\!-\!\tfrac12)\tfrac\pi n)|^2\,|\cot((j\!-\!\tfrac12)\tfrac\pi n)|^2&=\Big(\sum_{i=1}^n|\cot((i\!-\!\tfrac12)\tfrac\pi n)|^2\Big)^2\\
&\leq4\Big(\sum_{i=1}^{n/2}|\cot((i\!-\!\tfrac12)\tfrac\pi n)|^2\Big)^2\leq n^4,
\end{align*}
to obtain
\begin{align*}
\|M_3\|\leq \tfrac{1}{4n}\max_{ij}|f_{ij}|.
\end{align*}
Using the notation $\|E\|_{\infty}:=\max\{\|Eu\|_{\infty}\,; u\in\C^n,\|u\|_\infty\leq1\}$ for the operator norm induced by the maximum norm in $\C^d$, it holds $\|E\|_\infty=n/4$, whence $\|\tfrac{1}{2n^{3/2}}E\|_\infty\leq1/8$ and thus $\|(1+\tfrac{1}{2n^{3/2}}E)^{-1}\|_\infty\leq\tfrac{1}{1-1/8}=\tfrac87$.
This shows that
\begin{align*}
\max_{ij}|f_{ij}|\leq \|(1+\tfrac{1}{2n^{3/2}}E)^{-1}\|_\infty\max_{ij}|e^2_{ij}|\leq \frac{2n}{7},
\end{align*}
by denoting $e^2_{ij}$ the entries of the matrix $E^2$and noticing that 
$\max_{i,j}|e^2_{ij}|=n/4$. Finally, we obtain $\|M_3\|\leq \tfrac{1}{14}$ and
$\|M_1+M_2+M_3+M_4\|\leq\tfrac34+\tfrac18+\tfrac{1}{14}+\tfrac{1}{56}<1$.\bigskip

\bigskip

{\bf Acknowledgement.} The authors thank the referees for a careful reading of the article and for their useful remarks 
and suggestions. The first-named author was supported in part by ANR Projet Blanc DYNOP.

\bibliographystyle{amsplain}

\medskip

\noindent C. B. : Laboratoire Paul Painlev\' e, UMR no. 8524, B\^at. M2,
Universit\'e Lille 1, 59655 Villeneuve d'Ascq Cedex, France. \\ E-mail : {\tt badea@math.univ-lille1.fr}

\smallskip

\noindent M. C. : Institut de Recherche Math\'ematique de Rennes, UMR no. 6625,
Universit\'e Rennes 1, Campus de Beaulieu, 35042
Rennes Cedex, France \\ E-mail : {\tt michel.crouzeix@univ-rennes1.fr}
\end{document}